\input{style/arxiv-ba.cfg}
\documentclass[ba,linksfromyear,preprint]{imsart}
\usepackage{amssymb,amsmath,amsthm}
\usepackage{vtexurl}

\pubyear{2015}
\volume{10}
\issue{1}
\firstpage{237}
\lastpage{241}
\doi{10.1214/14-BA938}

\startlocaldefs
\newcommand{\enquote}[1]{``#1''}

\newtheorem{lemma}{Lemma}
\let\bid@start\@empty
\let\bid@end\@empty

\def\MR@url{http://www.ams.org/mathscinet-getitem?mr=}
\def\MR#1{\href{\MR@url#1}{MR#1}}
\def\BDOI#1{%
\edef\doi@base@i{\doi@base}\def\doi@base{}%
doi:~\doiurl{\doi@base@i#1}}

\appto\bid@start{\def\doi@size{\ttfamily}}
\appto\bid@end{\unskip.}

\define@key{bid}{mr}{\def\bid@mr{\MR{#1}}}
\define@key{bid}{doi}{\def\bid@doi{\BDOI{#1}}}
\define@key{bid}{pubmed}{}
\define@key{bid}{pii}{}
\define@key{bid}{issn}{}

\def\bid#1{%
       \bgroup
       \bid@start
       \let\bid@output\@empty
       \setkeys{bid}{#1}\ignorespaces%
       \ifdefvoid\bid@mr{}{\appto\bid@output{\bid@mr}}%
       \ifdefvoid\bid@doi{}{
         \ifdefempty\bid@output{}{\appto\bid@output{. }}
         \appto\bid@output{\bid@doi}%
       }%
       \bid@output
       \bid@end
       \egroup
}
\endlocaldefs

\begin{document}

\begin{frontmatter}
\title{Comment on Article by Berger, Bernardo, and~Sun\thanksref{T1}}
\runtitle{Comment on Article by Berger, Bernardo, and Sun}

\relateddois{T1}{Main article DOI: \relateddoi[ms=BA915,title={James O. Berger, Jose M. Bernardo, and Dongchu Sun. Overall Objective Priors}]{Related item:}{10.1214/14-BA915}.}

\begin{aug}
\author[a]{\fnms{Gauri Sankar} \snm{Datta}\ead[label=e1]{gaurisdatta@gmail.com}}
\and
\author[b]{\fnms{Brunero} \snm{Liseo}\corref{}\ead[label=e2]{brunero.liseo@uniroma1.it}}

\runauthor{G. S. Datta and B. Liseo}

\address[a]{University of Georgia, Athens, GA (USA), \printead{e1}}
\address[b]{Sapienza Universit\`a di Roma, Italy, \printead{e2}}

\end{aug}

\end{frontmatter}


It is our distinct pleasure to comment on a very thought provoking
paper, and we
first congratulate the Authors for this new masterly contribution in
the field
of objective priors.

The main goal of the paper is to find a multi-purpose objective prior
for a model
that should be used by different researchers with varying goals,
with the consequence that no single parameter or parametric function
can be identified as a parameter of interest.
In this situation, the most popular approaches either fail or, as in
the case of
the reference prior algorithm, they cannot be used.

Three general methods are discussed by the Authors. The first one is
limited to a number of
particular situations where the reference prior is the same for all
quantities of interest:
this case is not of much concern since a natural solution exists.
The second method is based on the reference prior approach: one looks
for the
prior which produces the marginal posteriors for
the quantities of interest which are closer -- in some sense --
to the marginal reference posteriors. Whereas this method
is perfectly reasonable, the final result will depend on the particular set
of the quantities of interest considered and it cannot be considered
as the ``overall'' objective prior.
The third method is based on a hierarchical representation of the
model, when it is available.
It shifts the problem of determining an objective prior
to an upper level of the hierarchy, where the impact of the prior might be
less serious.

We believe that the latter method is superior to the others because
\begin{itemize}
\item it is compatible with a predictive approach where
all the parameters are nuisance parameters and there is no particular
quantity of interest;
however, one should be careful here: if the quantity of interest is,
for example, the posterior
predictive mean
\[
E(X_{n+1}\mid X_1, \dots, X_n)
\]
of a future observation 
-- and not the entire predictive density -- then a parameter of
interest actually does exist!
\item it is clearly superior to Method 2, especially when the model is
used repeatedly
by different people which are interested in different sets of parameters.
\end{itemize}

In terms of prediction, it would be worth discussing
the proposal of \cite{Datta00}.

In this contribution, we will briefly consider the multinomial example,
and provide
some comments on the concept of prior averaging.

\section{The multinomial model in the sparse case}

This is a very interesting problem. Jeffreys' prior allocates a weight
of $1/2$ to each original
component of the vector $(\theta_1, \theta_2, \dots, \theta_m)$. This
is too much when
$m$ is large compared to the sample size $n$ and the distribution is
very sparse.
This suggests that the prior mass should be adequately spread
on the parameter space in such a way that each cell has a negligible
prior mean, especially when compared with the weight of the data.

In the multinomial case, the prior weight (expressed as the sum of the
hyper-parameters of the Dirichlet prior)
is equal to $m/2$ for the Jeffreys' prior, while in the hierarchical
approach, arising from a
Dirichlet$(a,a,\dots, a)$ hyper-prior,
it is a random quantity $v=ma$ with density given by expression (25) of
the paper,
at least in the case of an infinite $m$.
Several numerical computations, with different values of $n$ and $r_0$
(i.e.,~the number of non-empty cells),
show that the mode and the median of $v$ are rarely larger than 2,
so the hierarchical approach automatically accounts for the sparsity
and the
corresponding marginal posteriors are dramatically different from those
arising from the use of Jeffreys' prior.

There are many ways in which this problem can be handled.
If we transform it to a multiple testing problem, that is, for each
cell $i$ we test
\[
H_0 : \theta_i = 0 \qquad\mbox{ vs. } \qquad H_1: \theta_i \not= 0,
\]
the problem can be rephrased as that of finding an
ad-hoc prior, just like in the sparse normal problem, which
is well studied in literature, see, for example, \cite{scott}.
The two problems are similar but not identical: here we do not necessarily
observe data for each cell, and the difficulties associated with this
discrete version of the
problem are even greater
since the values of the $\theta_i$'s will affect the standard deviation
of the cells,
not only the means.

From a testing perspective there is also another interesting connection:
the Authors propose to add -- as a prior weight -- something close to
$1/m$ to each cell.
So the total weight of the prior will be approximately one. This
reminds us of the unit prior information of \cite{kw}.

The sparse multinomial case is also of theoretical interest because it
represents a
bridge between parametric and non-parametric models, when the number of
cells goes to infinity.

Our personal view of the example is close to that of the Authors,
although it is not of
great surprise that the Jeffreys' prior does not clearly discriminate
between observed and non-observed cells,
when $n$ is so small compared to $m$.
In other words, this is too much to ask of the prior.
When $n$ is as small as 3, and the number of parameters is about 1000,
it is hopeless to find a good automatic objective
prior and some external guidance (in this case, the choice of a
``proper'' prior
within the Dirichlet class) seems unavoidable.

More interesting is the fact that the hierarchical prior depends on $m$
and $n$ only through
their ratio: this is actually what one would expect.

We have also considered a variant of the multinomial example. In
particular, we have considered the case when the multinomial likelihood
can be rephrased as one arising from a sample of $m$ independent
Poisson random variables with mean vector ($\psi_1, \dots, \psi_m)$ and
then setting $\theta_j= \psi_j/\sum_i \psi_i$. Doing the usual
reference prior calculations here, we ended up with the same
conclusions as if we have used the standard Jeffreys' Beta prior
$(1/2,1/2)$ for the $\theta_i$'s. We wonder how to get the same result
(weights $\approx m^{-1}$ for the cells) in this alternative
perspective. It is very likely that this can be obtained by assuming
independent gamma priors with shape parameter $a$ and scale parameter
$\beta$ for the $\psi_j$. If the ``nuisance'' scale parameter $\beta$
is eliminated by conditioning on the total counts, we end up with the
same conclusion. However, the rationale behind this last choice is --
again -- only pragmatic.

A related issue is the ordered multinomial example in \xch{Section }{\S}2.1.2.
Here the overall prior for any of the parameters $(\xi_1, \dots, \xi
_m)$ is the product of independent Beta$(1/2, 1/2)$:
what happens for large $m$? Is the overall prior still a sensible prior
or should we take into account this problem?

\section{A comment on geometric average of priors}
Consider the following divergence function
\[
d(\eta) = \sum_{i=1}^m \alpha_i \int\eta(\theta) \log\frac{\eta(\theta
)}{\pi_i(\theta)} d\theta,
\]
where $\alpha_1, \dots, \alpha_m \geq0$ are suitable constants adding
to 1, and $\pi_i(\theta)$ may be a suitable objective prior when one is
interested in one of a given set of $m$ parametric functions.
The above function is a weighted average Kullback--Leibler divergence
between a global prior and the marginal priors we would like to use in
the case we were interested in a single parametric function $t_i(\theta
)$, $i=1, \dots, m.$
Note that
\begin{align*}
d(\eta) &= \int\eta(\theta) \log\eta(\theta) d\theta- \sum_{i=1}^m
\int\eta(\theta)\log
\pi^{\alpha_i}_i(\theta) d\theta\\
&= \int\eta(\theta) \log\frac{\eta(\theta)}{\prod_{i=1}^m \pi^{\alpha
_i}_i(\theta)}d\theta.
\end{align*}
By Jensen's inequality, $d(\eta)$ will be minimized with respect to
$\eta$ if
$\eta(\theta)/\prod_{i=1}^m \pi^{\alpha_i}_i(\theta)$ is a degenerate function.
This leads to the geometric mean prior
\[
\pi_G(\theta)\propto\prod_{i=1}^m \pi^{\alpha_i}_i(\theta).
\]

Usually, the component priors $\pi_i(\theta)$'s are improper, which in
turn may also make
$\pi_G(\theta)$ an improper prior. The authors indicated that the
geometric mean prior is preferable to the arithmetic mean prior since
one or more of the component priors $\pi_i$ may be improper, and the
arithmetic mean posterior may be highly influenced by one or a few
component posteriors. Indeed, for any arbitrary positive constant
$c_i$, $c_i\pi_i(\theta)$ is as much an objective prior as $\pi_i(\theta
)$ is. While the posterior propriety of the arithmetic mean prior is an
immediate consequence of the propriety of the component posteriors, the
same is not so obvious for the geometric mean prior.
However, the following lemma shows that the posterior corresponding to
$\pi_G(\theta)$ will be proper provided that each component prior $\pi
_i(\theta)$ generates a proper posterior.

\begin{lemma}
For two prior densities $\mu(\theta)$ and $\nu(\theta)$, if
\[
\int\mu(\theta) L(\theta; \mathbf{x}) d\theta< \infty, \quad \makebox
{ and }
\quad \int\nu(\theta) L(\theta; \mathbf{x}) d\theta< \infty,
\]
%
then, for any $\alpha\in(0,1),$
\[
\int\mu^\alpha(\theta) \nu^{1-\alpha}(\theta) L(\theta; \mathbf{x})
d\theta< \infty,
\]
where $L(\theta; \mathbf{x})$ denotes the joint density of data $\mathbf
{x}$ corresponding to the parameter value~$\theta$.
\end{lemma}

\begin{proof}
By H\"{o}lder's inequality, it follows that
\begin{align*}
\int\mu^\alpha(\theta) \nu^{1-\alpha}(\theta) L(\theta; \mathbf{x})
d\theta
&=
\int\left[\mu(\theta)L(\theta; \mathbf{x})\right]^\alpha
\left[\nu(\theta)L(\theta; \mathbf{x})\right]^{1-\alpha}
d\theta\\
&\leq
\left[ \int\mu(\theta) L(\theta; \mathbf{x}) d\theta\right]^\alpha
\left[ \int\nu(\theta) L(\theta; \mathbf{x}) d\theta\right
]^{1-\alpha} .
\end{align*}
Thus $\mu(\theta)^\alpha\nu(\theta)^{1-\alpha}$ generates a proper
posterior density for the given data $\mathbf{x}$.
\end{proof}
By repeated use of this lemma, the propriety of the posterior based on
the geometric prior $\pi_G(\theta)$
easily follows.

\section{An anecdote}
While preparing the present comments one of the authors attended a seminar
on applied probability where the following situation was presented.
In a small village, there is a chief and several shepherds. Each
shepherd runs
a flock of sheeps.
The chief knows that the ground of their village is going to become
parched so
the shepherds have to move away.
All the roads starting from the village -- but one -- are full of
hungry wolves.
The chief has his own probability distribution about which is the safe road.
If the chief communicates his/her information to the shepherds, it is
very likely
that all of them would choose the same road. This implies that either
all the sheeps or none will survive. If the chief does not communicate
his/her information, it is likely that the shepherds will randomly
choose the road.

The question is: should the chief share this information with the
shepherds or not?
If so, (s)he is playing a risky (all or nothing) strategy. If not,
(s)he is
taking a minimax strategy where it is more likely that some of the
flocks will survive.
Is there a way to calibrate the amount of information to be shared?

There are several interesting similarities between this story and the
main issue of the paper.
Is there a way to find a compromise between the general goal and a single
objective? Is it possible to find a prior -- or a strategy -- which is
not so bad
for any of the problems at hand?

Our view is that, if the answer is ``yes'', this prior should not depend
on the particular list of problems. In other words, it would be great
to have just ``one''
overall prior. In this respect, the hierarchical approach seems to be
more promising.

\bibliographystyle{ba}

\end{document}